\documentclass[10pt,doublespace]{article}
\usepackage{psfrag}
\usepackage{amssymb,bbm}
\usepackage{amsmath}
\usepackage{amsfonts}
\usepackage{pstricks,pstricks-add}
\usepackage{graphicx}
\usepackage[latin5]{inputenc}
\usepackage{yfonts}
\usepackage{setspace}
\usepackage[toc,page]{appendix}
\newcounter{theorem}
\newtheorem{theorem}{Theorem}

\newtheorem{lemma}{Lemma}

\newtheorem{example}{Example}

\newtheorem{remark}{Remark}
\newtheorem{definition}{Definition}
\newtheorem{corollary}{Corollary}
\newenvironment{proof}[1][Proof]{\textbf{#1.} }{\rule{0.5em}{0.5em}}

\begin{document}
\title{On the Minkowski Measurability \\ of Self-Similar Fractals in $\mathbb{R}^d$ \date{}}

\author{Ali DENİZ   \footnote{Anadolu
University, Science Faculty, Department of Mathematics, 26470,
Eskişehir, Turkey,
\hspace{5cm}e-mails:$\quad$
adeniz@anadolu.edu.tr, skocak@anadolu.edu.tr, yunuso@anadolu.edu.tr,
aeureyen@anadolu.edu.tr}
 \and Şahin KOÇAK $^*$ \and Yunus ÖZDEMİR $^*$ \and Andrei V. RATİU
\footnote{Department of Mathematics, Istanbul Bilgi University, Istanbul, Turkey, \hspace{5cm}e-mail:$\quad$
ratiu@bilgi.edu.tr }
  \and A. Ersin
ÜREYEN $^*$ }


\maketitle
\begin{abstract}
M. Lapidus and C. Pomerance (1990-1993) and K.J. Falconer
(1995) proved that a self-similar fractal in $\mathbb{R}$ is Minkowski-measurable iff it is of non-lattice type. D. Gatzouras (1999) proved that a self-similar fractal in $\mathbb{R}^d$ is Minkowski measurable if it is of non-lattice type (though the actual computation of the content is intractable with his approach) and conjectured that it is not Minkowski measurable if it is of lattice type. Under mild conditions we prove this conjecture and in the non-lattice case we improve his result in the sense that we express the content of the fractal in terms of the residue of the associated $\zeta$-function at the Minkowski-dimension.
\end{abstract}


\section{Introduction}
Let \[F=\bigcup \limits_{j=1}^J \varphi_j(F)=:\Phi(F) \subset \mathbb{R}^d\] be a self-similar fractal, where $\varphi_j:\mathbb{R}^d \to \mathbb{R}^d $ are similitudes with scaling ratios $0<r_j<1,\, j=1,2,\dots,J,$ for $J\geq 2$. We assume the iterated function system (IFS) $\Phi$ to satisfy the open set condition, so that the Minkowski dimension $D$ of $F$ is given by the unique real root of the Moran equation $\sum_{j=1}^J r_j^D=1.$

Let $F_{\varepsilon} =\{ x\in \mathbb{R}^d \,|\, {\rm dist}(x,F)\leq \varepsilon\}$ and $V_F(\varepsilon)$ be the $d$-dimensional volume of $F_{\varepsilon}$. $F$ is called Minkowski measurable if the limit \[\mathcal{M}(F):=\lim \limits_{\varepsilon \to 0^+} V_F(\varepsilon) \, \varepsilon^{D-d}\] exists, is finite and different from zero. $\mathcal{M}(F)$ is then called the Minkowski content of $F$.

The IFS $\Phi$ is called of lattice type, if the additive subgroup $\sum_{j=1}^J (\log r_j) \mathbb{Z}$ of $\mathbb{R}$ is discrete and otherwise (i.e. if this subgroup is dense in $\mathbb{R}$) of non-lattice type (see \cite{LaFra}). In the lattice case there is an $r$ with $\log r_j = k_j \log r$, $k_j \in \mathbb{Z}^+$. This dichotomy is decisive for Minkowski measurability of fractals as shown for the one-dimensional case by Lapidus-Pomerance \cite{LaPom1},\cite{LaPom2} and  Falconer \cite{falconerpro}. We now briefly recall their results:

Let $d=1$ and $I$ denote the convex hull of $F$, $I=[F]$. By the open set condition, $\varphi_j (I)$ and $\varphi_k(I)$ are disjoint for $j \neq k$, except possibly at the endpoints. There will emerge $Q \leq J-1$ gaps  on $I$, with lengths $l_q$, $q=1,2,\dots, Q$. Then, $F$ is Minkowski measurable if and only if the IFS $\Phi$ is of non-lattice type, and in case it is measurable the content is given by
\begin{equation} \label{nonlatiscont}
\mathcal{M}=\frac{2^{1-D} \sum_{q=1}^Q l_q^D}{D(1-D) \sum_{j=1}^J r_j^D \log r_j^{-1}} \quad \text{ (see \cite[p.262]{LaFra})}.
\end{equation}
In the lattice case the fractal is not Minkowski measurable, but one can define an average Minkowski content by the formula
\begin{equation}\label{averajcont}
\mathcal{M}_{{\rm av}}=\lim \limits_{T \to \infty} \frac{1}{\log T} \int_{1/T}^1 \varepsilon^{-(1-D)} V_F(\varepsilon) \frac{d\varepsilon}{\varepsilon}  \quad \text{ (see \cite[p.257]{LaFra})}
\end{equation}
and the formula (\ref{nonlatiscont}) gives in this case the average Minkowski content. (In $\mathbb{R}^d$, the exponent $-(1-D)$ of $\varepsilon$ should be replaced by $-(d-D)$).
\begin{figure}[ht]
\centering
\includegraphics[scale=0.75]{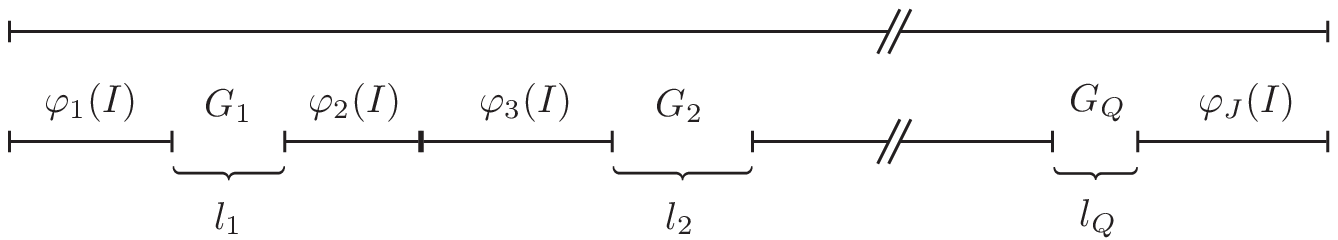}\caption{}
\end{figure}
In higher dimensions, Gatzouras proved in \cite{gatzouras} that a non-lattice self-similar fractal in $\mathbb{R}^d$ (satisfying the open set condition) is Minkowski measurable. He used renewal theory and the formula he gave for the Minkowski content is hardly usable for explicit computations. For lattice case he conjectured that the Minkowski content does not exist. Lapidus and van Frankenhuijsen \cite[Remark 12.19]{LaFra} remarked that renewal theory was unlikely to yield this result but their approach by higher dimensional tube formulas would settle this issue. Our aim in this paper is in a certain sense to carry out this program under some mild additional conditions on the IFS we give below. Using the Lapidus-Pearse theory \cite{lape1} we will give an alternative proof of the existence
of the Minkowski content in non-lattice case with a very explicit and applicable formula for the content and we will prove the non-Minkowski measurability in the lattice case.

Let $C=[F]$ be the convex hull of $F$, for which we assume ${\dim C}=d$. Adopting the approach of Pearse-Winter \cite{pewin} we want to put some additional conditions on the IFS $\Phi$:
\begin{enumerate}
\item[\textbf{TSC}] (Tileset Condition): $\Phi$ satisfies the open set condition with ${\rm int}\,C$ as a feasible open set.
\item[\textbf{NTC}] (Nontriviality Condition): ${\rm int}\,C \nsubseteq \Phi(C)=\bigcup \limits_{j=1}^J \varphi_j(C)$.
\end{enumerate}

Now define $T_0= {\rm C \setminus \Phi (C)}$ and its iterates $T_n=\Phi^{n} (T_0)$, $n=1,2,3, \dots $ (see \cite{pearse}). The tiling of the self-similar system is given by \[ \mathcal{T}:=\{T_n\}_{n=0}^{\infty}.\]
Let $V_{T_n}^-(\varepsilon)$ denote the volume of the inner $\varepsilon$-neighborhood of $T_n$ ( i.e. $\{x \in \overline {T_n} \ | \ {\rm dist} (x , T_n^{c}) \leq \varepsilon \} $ ) and $V_{\mathcal{T}} ^-(\varepsilon):= \sum \limits_{n=0}^{\infty} V_{T_n}^-(\varepsilon)$.

Pearse and Winter prove in \cite{pewin} the following implication: If the above conditions TSC and NTC hold, then the property $\partial C \subset F$ implies $V_F(\varepsilon)=V_{\mathcal{T}} ^-(\varepsilon)+V_C(\varepsilon)-V_C(0)$.
This is extremely important, because there are formulas available for $V_{\mathcal{T}} ^-(\varepsilon)$ (see below, Theorem \ref{ReziduThm}) and this relationship enables one to compute the true volume of the $\varepsilon$-neighborhood of the fractal. We will call this condition the Pearse-Winter condition:
\begin{enumerate}
\item[\textbf{PWC}] (Pearse-Winter Condition): $\partial C \subset F$.
\end{enumerate}

To state the tube formula we need some additional assumptions and definitions.
 Assume that $T_0$ is the union of finitely many (connected) components,
 $T_0=G_1 \cup G_2 \cup \cdots \cup G_Q$, called the generators of the tiling.  We assume the generators to be (diphase) Steiner-like in the following sense of Lapidus-Pearse \cite{lape1}:
 A bounded, open set $G \subset \mathbb{R}^d$ is called (diphase) Steiner-like
if the volume $V^-_G(\varepsilon)$ of the inner
$\varepsilon$-neighborhood of $G$ admits an expression of the form
\begin{equation}\label{polinom}
V^-_{G}(\varepsilon)= \textstyle \sum_{m=0}^{d-1} \kappa_m(G)\varepsilon^{d-m},\qquad \text{ for } \varepsilon < g,
\end{equation}
where  $g$ denotes the  inradius of $G$, i.e. supremum of the radii
of the balls contained in $G$. For $\varepsilon \geq g$ we have
$V_{G}^-(\varepsilon)={\rm{volume}}(G)$ which is denoted by
$-\kappa_d(G)$, the negative sign being conventional \cite{lape1}.

Lapidus-Pearse introduce the following ``scaling
$\zeta$-function'':
\begin{definition}\label{scalingzeta}
The scaling $\zeta$-function of the self-similar fractal is
defined by
\[\zeta(s)=\sum_{n=0}^\infty \sum_{w \in W_n}r_w^s,\] where $W_n$
is the set of words $w=w_1w_2\cdots w_n$ of length $n$ (with
letters from $\{1,2,\dots ,J\}$) and $r_w=r_{w_1}r_{w_2}\dots
r_{w_n}$.
\end{definition}

The above series can be shown to converge for ${\rm {Re}} (s)> D$.
A simple calculation shows that $\zeta(s)$ can be
expressed as \cite[Theorem 2.4]{LaFra}
\begin{equation*}
\zeta(s)=\frac{1}{1-\sum_{j=1}^Jr_j^s} \quad \text{
for } {\rm {Re}} (s)> D.
\end{equation*}
$\zeta(s)$ can then be meromorphically extended to the whole complex plane. We will denote this extension also by $\zeta(s)$.

\begin{definition}
The set $\mathfrak{D}:=\{ \omega \in \mathbb{C} \mid \zeta (s)  \text {
has a pole at } \omega \}$ is called the set of complex dimensions
of the self-similar fractal.
\end{definition}

Lapidus-Pearse define a second type of ``$\zeta$-functions'' associated
with the tiling and related to the geometry of the (diphase) Steiner-like
generators. We assume for simplicity that there is a single generator $G$ (so that $T_0=G$).

\begin{definition}
 The geometric $\zeta$-function $\zeta_{\mathcal{T}}(s,\varepsilon)$ associated with the generator $G$ is defined by \[\zeta_{\mathcal{T}}(s,\varepsilon):=\zeta(s)\, \varepsilon ^{d-s}  \sum_{m=0}^d \frac{g^{s-m}}{s-m}\kappa_m(G).\]
\end{definition}
We now state the formula of Lapidus-Pearse for
$V^-_{\mathcal{T}}(\varepsilon)$ (though we will use a modification of this formula in the proof of our higher dimensional content formulas below):
\begin{theorem}[Tube formula for tilings of self-similar fractals, \cite{lape1}]
\begin{equation*}
V^-_{\mathcal{T}}
(\varepsilon)=\sum_{\omega \in \mathfrak{D}_{\mathcal{T}}}
{\rm{res}} (\zeta_{\mathcal{T}}(s,\varepsilon);\omega),
\end{equation*}
where $\mathfrak{D}_{\mathcal{T}}=\mathfrak{D} \cup \{0,1,\dots,d-1\}$. \label{ReziduThm}
\end{theorem}
\begin{remark}
Lapidus-Pearse gives in \cite{lape1} a distributional proof for this formula. For a pointwise proof see \cite{bizim}.
\end{remark}
\section{Main Results}
Our main result is the following theorem:
\begin{theorem}\label{thmcontent}
Let $F=\Phi(F)=\bigcup \limits_{j=1}^{J} \varphi_j(F)$ be a self-similar fractal in $\mathbb{R}^d$ with ${\rm dim} [F]=d$ and the contractivity ratios of the similitudes $\{\varphi_j\}$ being $\{r_j\}$.

We assume the Tileset Condition, the Nontriviality Condition and the Pearse-Winter Condition to hold (see TSC, NTC and PWC in the former section). We additionally assume $D > d-1$, where $D$ is the Minkowski dimension of $F$.

Under these assumptions the followings hold:
\begin{enumerate}
\item[I.] If the IFS $\Phi$ is of non-lattice type, then $F$ is Minkowski measurable with Minkowski content
\begin{eqnarray*}
\mathcal{M}(F)={\rm res}\left ( \zeta_{\mathcal{T}}(s,\varepsilon)\,\varepsilon^{s-d};D \right )&=&{\rm res}\left (\zeta(s);D \right ) \sum \limits_{m=0}^d \frac{g^{D-m}}{D-m}\kappa_m(G)\\
&=&\frac{\sum\limits_{m=0}^d \frac{g^{D-m}}{D-m}\kappa_m(G)}{\sum\limits_{j=1}^J  r_j^D \log r_j^{-1} },
\end{eqnarray*}
\item[II.] If the IFS $\Phi$ is of lattice type then $F$ is not Minkowski measurable. The average Minkowski content as defined in (\ref{averajcont}) exists and equals \[\mathcal{M}_{{\rm av}}(F)= {\rm res} \left ( \zeta_{\mathcal{T}}(s,\varepsilon) \, \varepsilon^{s-d};D \right ).\]
\end{enumerate}
\end{theorem}

\begin{remark}
In case of multiple generators, let $\zeta_{\mathcal{T}}^q(s,\varepsilon) \,\,(q=1,2,\dots,Q)$ denote the geometric zeta function corresponding to generator $G_q$ and define the total geometric zeta function by
\begin{eqnarray*}
\zeta_{\mathcal{T}}(s,\varepsilon) = \sum \limits_{q=1}^Q \zeta_{\mathcal{T}}^q(s,\varepsilon)&=& \zeta(s) \,\varepsilon^{d-s} \sum\limits_{q=1}^Q  \sum \limits_{m=0}^d \frac{g_q^{s-m}}{s-m}  \kappa_m(G_q),
\end{eqnarray*}
where $g_q$ denotes the inradius of $G_q$. Then
\begin{enumerate}
\item[I.] If the IFS $\Phi$ is of non-lattice type then $F$ is Minkowski measurable with Minkowski content \[
\mathcal{M}(F)={\rm res} \left ( \zeta_{\mathcal{T}}(s,\varepsilon) \, \varepsilon^{s-d};D \right )=\frac{ \sum \limits_{q=1}^Q  \sum \limits_{m=0}^d \frac{g_q^{D-m}}{D-m}  \kappa_m(G_q)}{\sum\limits_{j=1}^J  r_j^D \log r_j^{-1} },
\]
\item[II.] If the IFS $\Phi$ is of lattice type then $F$ is not Minkowski measurable and the average Minkowski content exists and equals
    \[\mathcal{M}_{{\rm av}}(F)= \sum \limits_{q=1}^Q {\rm res} \left ( \zeta_{\mathcal{T}}^q(s,\varepsilon) \, \varepsilon^{s-d};D \right ).\]
\end{enumerate}
\end{remark}

\begin{corollary}
If we specialize to the dimension $d=1$, we obtain the formula (\ref{nonlatiscont}): Since each $G_q$ is an interval of length $l_q$ we have \[ V^-_{G_q}(\varepsilon)=\left \{
\begin{array}{cc}
 2 \varepsilon  &,\text{ for }  \, \varepsilon < g_q \\
 l_q &, \text{ for }\, \varepsilon \geq g_q,
\end{array}
\right .\] so that $\kappa_0(G_q)=2, \, \kappa_1(G_q)=-l_q$ and the inradius $g_q$ equals $l_q/2$. Then
\[\zeta_{\mathcal{T}}^q(s,\varepsilon) = \zeta(s)\, \varepsilon^{1-s} \left ( 2 \frac{g_q^s}{s}-l_q \frac{g_q^{s-1}}{s-1}\right ) = \zeta(s)\,  \varepsilon^{1-s}  \frac{2^{1-s} l_q^s}{s(1-s)},\]
and  \[\mathcal{M}(F)={\rm res} \left ( \zeta(s)\,\frac{2^{1-s} }{s(1-s)}\sum \limits_{q=1}^Q l_q^s  ;D\right )= \frac{ 2^{1-D} \sum \limits_{q=1}^Q l_q^D }{D(1-D) \sum \limits_{j=1}^J  r_j^D \log r_j^{-1} }. \]
\end{corollary}
\begin{example}
Let $\triangle ABC$ be an acute triangle with corresponding sides $a,b$ and $c$. Let $\triangle A'B'C'$ be its orthic (pedal) triangle  (see Fig.\ref{Fig:ex1}a).
The triangles $\triangle AC'B'$, $\triangle BA'C'$ and $\triangle CB'A'$ are scaled copies of the original triangle $\triangle ABC$ with scaling ratios $\cos A, \cos B$ and $\cos C$ (denoting the angle at the vertices $A,B,C$ again with the same letter).
Consider the collection of these maps as an iterated function system $\Phi=\{\varphi_j\}_{j=1}^{3}$ on $\mathbb{R}^2$  as indicated in Fig.\ref{Fig:ex1}b. The associated self-similar fractal (``orthic fractal") is shown in Fig.\ref{Fig:ex1}c (see \cite{zhang}). The Minkowski dimension $D$ is determined by $(\cos A)^D+(\cos B)^D+(\cos C)^D=1$. This system satisfies the TSC, NTC and PWC, and has a single generator $G=\triangle A'B'C'$.

The volume of the inner $\varepsilon$-neighborhood of the generator $G$ is given by
\begin{equation*}
V_G(\varepsilon) = \left \{
\begin{array}{ccc}
  \kappa_1(G) \varepsilon+ \kappa_0(G) \varepsilon ^2 & , & \text{ for } \varepsilon \leq g \\ \\
 -\kappa_2(G) & , & \text{ for }\varepsilon \geq g
\end{array}
\right .
\end{equation*}
where \[g=\frac{4 {\rm Area}(\triangle ABC) \cos A \cos B\cos C }{a \cos A+b \cos B +c \cos C},\]
and $\kappa_0(G)=-(\tan A+\tan B+\tan C),\kappa_1(G)=a \cos A + b \cos B+ c \cos C$, $\kappa_2(G)=-2 \cos A \cos B\\ \cos C {\rm Area}(\triangle ABC)$.
\begin{figure}[ht]
\centering
\includegraphics[height=4.5cm]{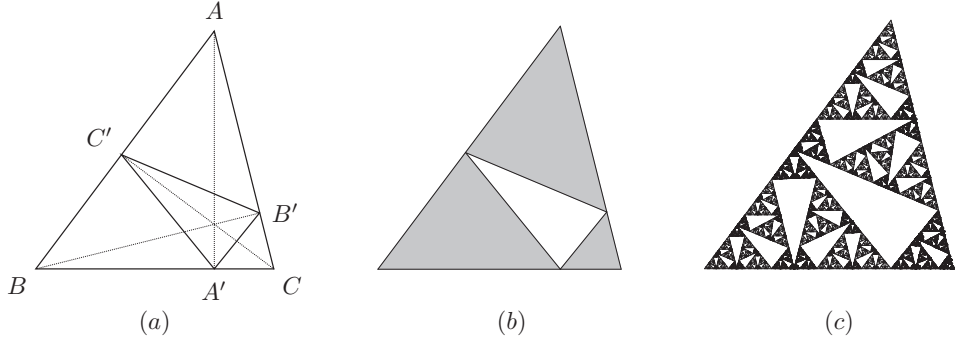}
\caption{($a$) An acute triangle $\triangle ABC$ with its orthic triangle $\triangle A'B'C'$, ($b$) the IFS, ($c$) its attractor.}\label{Fig:ex1}
\end{figure}
Depending on the angles $A,B,C$, the orthic fractal may be of lattice or non-lattice type. If it is of non-lattice type, then by Teorem \ref{thmcontent}.I, its Minkowski content exists and is given by
\[-\frac{\left( \frac{g^{D}}{D} \kappa_0(G)+\frac{g^{D-1}}{D-1} \kappa_1(G)+\frac{g^{D-2}}{D-2} \kappa_2(G)\right)}{(\cos A)^D \log (\cos A) + (\cos B)^D \log (\cos B)+(\cos C)^D \log (\cos C)}.\] If the orthic fractal is of lattice type then, by Teorem \ref{thmcontent}.II, the Minkowski content does not exist, but the average content exists and is given by the same expression.
\end{example}
\section{Proof of Theorem \ref{thmcontent}} We consider first the more difficult non-lattice case (because in that case the distribution of the poles of the $\zeta$-function could be utterly complicated). By the assumptions of the theorem we have \[V_F(\varepsilon) = V_{\mathcal{T}}^- (\varepsilon) + V_C(\varepsilon) -V_C(0), \text{ where } C=[F].\] We have to consider the limit behaviour of $V_F(\varepsilon) \,\varepsilon ^{D-d}$ as $\varepsilon $ tends to zero.

By the well-known Steiner formula, the volume of the $\varepsilon$-neighborhood of a bounded convex set in $\mathbb{R}^d$ can be expressed as a polynomial in $\varepsilon$ \cite{schneider}:
\[V_C(\varepsilon) =\textstyle \sum_{m=0}^d a_m \varepsilon^m \quad \text{ with } a_0 = V_C(0).\] Hence, $\lim \limits_{\varepsilon\to 0^+} (V_C(\varepsilon)-V_C(0))\, \varepsilon^{D-d}=0$ (by the assumption $D>d-1$). Thus, our concern will be the term $V_{\mathcal{T}}^-(\varepsilon)\,\varepsilon^{D-d}$.

To make the proof transparent, we will formulate several lemmas, whose proofs we defer to the next section.

\begin{lemma}\label{yunus} For any $c$ satisfying $D < c < d$,
\[
V_{\mathcal{T}}^{-} (\varepsilon)=\frac{1}{2 \pi {\bf i}}\int_{c-\bf{i} \infty}^{c+\bf{i} \infty}  \zeta_{\mathcal{T}}(s,\varepsilon) ds.
\]
\end{lemma}

Now we want to convert this integral into an appropriate sum of residues of $\zeta_{\mathcal{T}}(s,\varepsilon)$ plus an integral on a path $\Gamma$ lying to the left of the line ${\rm Re}(s)=D$.

For the construction of this path $\Gamma$ we need the following two lemmas. For convenience we assume that the contractivity ratios are ordered as \[1>r_1\geq r_2 \geq \cdots \geq r_J>0.\]
\begin{lemma}\label{strip}
There exists $\widetilde D < D$ such that all  the poles of $\zeta(s)$ in the strip $\{s \ |\ \widetilde D <{\rm Re}(s) < D \}$ are simple and the residues of $\zeta(s)$ at these poles are bounded by $1/\log r_1^{-1}$.
\end{lemma}

\begin{lemma}\label{screen}
There exist strictly increasing, real sequences $\{\alpha_k\}_{k \in \mathbb{Z}}$ and $\{\beta_k\}_{k \in \mathbb{Z}}$ with $\alpha_k < \beta_k < \alpha_{k+1}$ for all $k$, $\alpha_0 < 0 < \beta_0$ and \[\alpha_{k+1}-\alpha_{k} > \frac{\pi}{\log r_J^{-1}}, \qquad (k \in \mathbb{Z}) \] and there exist $\sigma_L, \, \sigma_R$ with $\max\{\widetilde D,d-1 \} <\sigma_L<\sigma_R< D$, such that \\ $\zeta(s)$ is uniformly bounded for all $k\in\mathbb{Z}$ on the (oriented) segments
\begin{eqnarray*}
\begin{array}{ll}
  \gamma_k^1:=[\sigma_R +{\bf i} \beta_{k-1},\sigma_R +{\bf i} \alpha_{k}], & \gamma_k^2:=[\sigma_R +{\bf i} \alpha_{k},\sigma_L +{\bf i} \alpha_{k}], \\
  \gamma_k^3:=[\sigma_L +{\bf i} \alpha_{k},\sigma_L +{\bf i} \beta_{k}], & \gamma_k^4:=[\sigma_L +{\bf i} \beta_{k},\sigma_R +{\bf i} \beta_{k}].
\end{array}
\end{eqnarray*}
\end{lemma}
\begin{figure}[ht]
\begin{center}
\includegraphics[scale=0.85]{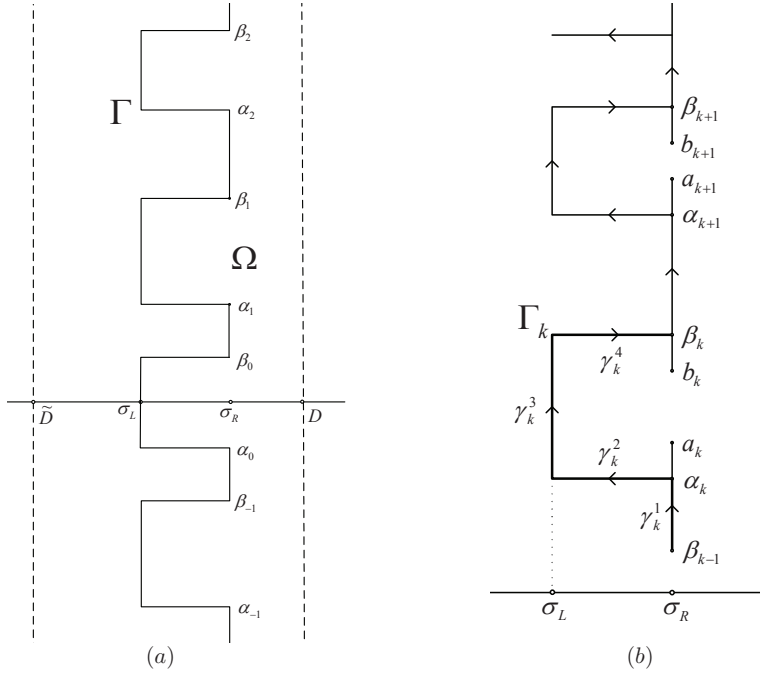}
\end{center}
\caption{The path $\Gamma$} \label{Fig:gamma}
\end{figure}

Let $\Gamma_k$ be the concatenation of the segments $\gamma_k^l,\, l=1,2,3,4$ and $\Gamma$ be the path obtained by the concatenation of all $\Gamma_k,\ k \in \mathbb{Z}$. Let $\Omega$ be the open region between $\Gamma$ and the line ${\rm Re} (s)=D$. Then, by Lemma \ref{strip} and Lemma \ref{screen}, there exists $K>0$ such that \[|\zeta(s)| \leq K, \, \text{ for all } s \in \Gamma\] and \[\left \vert {\rm res}(\zeta(s);\omega)\right \vert \leq K \ \text{ for all poles } \omega \in \Omega \text{ of } \zeta.\]

(As there are too many constants in the sequel, we will use the letter $K$ for any of them, though they may differ in the appearing context.)

As $\zeta(s)$ is analytic in $\{s\,|\, {\rm Re}(s)>D\}$, all the poles of $\zeta$ lie in the half plane $\{s\,|\, {\rm Re}(s) \leq D\}$, and by \cite[Theorem 2.17]{LaFra}, $D$ is the only pole of $\zeta$ with real part $D$. Now, the integral in Lemma \ref{yunus} can be expressed as follows:

\begin{lemma}
\label{CauchyL}
\[\frac{1}{2 \pi {\bf i}}\int_{c-{\rm i}\infty}^{c+{\rm i}\infty} \zeta_{\mathcal{T}}(s,\varepsilon)={\rm res}(\zeta_{\mathcal{T}} (s,\varepsilon);D)+\sum_{\omega \in \Omega \cap \mathfrak{D}} {\rm res}(\zeta_{\mathcal{T}} (s,\varepsilon);\omega) + \frac{1}{2 \pi {\bf i}}\int_{\Gamma}\zeta_{\mathcal{T}} (s,\varepsilon)ds. \]
\end{lemma}

The integral over $\Gamma$ on the right-hand side above is absolutely convergent and can be estimated as follows:
\begin{lemma}
\label{CauchyGamma}
$\displaystyle\int_{\Gamma} \vert \zeta_{\mathcal{T}} (s,\varepsilon)\vert \vert ds \vert = O(\varepsilon^{d-\sigma_R}) \text{ as } \varepsilon \to 0^+.$
\end{lemma}

This means that we will get rid of this term in the evaluation of the limit $V_{\mathcal{T}}^{-}(\varepsilon) \varepsilon^{D-d}$ as $\varepsilon \to 0^+$: $ O(\varepsilon^{d-\sigma_R}) \varepsilon^{D-d}=o(1) \text{ since } \sigma_R < D.$

We have
\[
\lim_{\varepsilon \to 0^+} {\rm res}(\zeta_{\mathcal{T}} (s,\varepsilon);D) \varepsilon^{D-d}=\frac{\sum\limits_{m=0}^d \frac{g^{D-m}}{D-m}\kappa_m(G)}{\sum\limits_{j=1}^J  r_j^D \log r_j^{-1} }
,\]
where the numerator of the right-hand side is different from zero by Remark \ref{remark2} below and the denominator is obviously non-zero. Therefore, the proof of first part of Theorem \ref{thmcontent} will be settled by the following lemma:

\begin{lemma}\label{lemma6}
$\lim\limits_{\varepsilon \to 0^+} \varepsilon^{D-d} \sum\limits_{\omega \in \Omega \cap \mathfrak{D}} {\rm res}(\zeta_{\mathcal{T}} (s,\varepsilon);\omega)=0$.
\end{lemma}

Now we consider the lattice case. In this case we can use simply a vertical line ${\rm Re}(s)=\sigma < D$ (with $\sigma$ sufficiently close to $D$) instead of the complicated $\Gamma$ and applying the same procedures we can arrive at the formula \[ V_{\mathcal{T}}^-(\varepsilon)=\sum_{n=-\infty}^{\infty} {\rm res}(\zeta_{\mathcal{T}} (s,\varepsilon);D+{\bf i} np) + \frac{1}{2 \pi {\bf i}}\int_{\sigma-{\bf i}\infty}^{\sigma+{\bf i}\infty}\zeta_{\mathcal{T}} (s,\varepsilon)ds, \]
where $p=2\pi/\log r$ (with $\log r$ being the generator of the group $\sum_{j=1}^J (\log r_j) \mathbb{Z}$). As in Lemma \ref{CauchyGamma}, the integral on the right-hand side is $O(\varepsilon^{d-\sigma})$ as $\varepsilon \to 0^+$, so that we can omit this term. The non-real complex dimensions emerging on the line ${\rm Re}(s)= D$ will now cause oscillations and prevent the function $\varepsilon ^{D-d}  V_{\mathcal{T}}^-(\varepsilon)$ to have a limit as $\varepsilon \to 0^+$:
\begin{eqnarray*}
\varepsilon ^{D-d}  V_{\mathcal{T}}^-(\varepsilon)&=&\frac{1}{\sum\limits_{j=1}^J  r_j^D \log r_j^{-1}} \sum_{n \in \mathbb{Z}} \varepsilon^{-{\bf i}np} \sum\limits_{m=0}^d \frac{g^{D+{\bf i}np-m}}{D+{\bf i} n p-m}\kappa_m(G)\\
&=:&\frac{1}{\sum\limits_{j=1}^J  r_j^D \log r_j^{-1}} \sum_{n \in \mathbb{Z}} a_n \varepsilon^{-{\bf i}np}=\frac{1}{\sum\limits_{j=1}^J  r_j^D \log r_j^{-1}} \sum_{n \in \mathbb{Z}}a_n e^{{\bf i}npx}
\end{eqnarray*}
by change of variable $x=-\log \varepsilon$. By \ref{polynom} below at most $d-1$ of $a_n$ can vanish and by \ref{toplamsinir} $\sum_{n \in \mathbb{Z}} |a_n| < \infty $, so that the above Fourier series uniformly converges, is non-constant and oscillates as $x \to \infty$ ($\varepsilon \to 0^+$). Thus a lattice fractal is never Minkowski measurable. However, the average Minkowski content always exists and can be calculated as follows (as $V_C(\varepsilon) - V_C(0)$  does not contribute):
\begin{eqnarray*}
\mathcal{M}_{\rm av} &=& \lim_{T \to \infty} \frac{1}{\log T} \int_{1/T}^1 \varepsilon^{D-d} V_F(\varepsilon) \frac{d\varepsilon}{\varepsilon}
= \lim_{T \to \infty} \frac{1}{\log T} \int_{1/T}^1 \varepsilon^{D-d} V_{\mathcal{T}}^-(\varepsilon) \frac{d\varepsilon}{\varepsilon} \\
&=& \lim_{T \to \infty} \frac{1}{\log T} \frac{1}{\sum\limits_{j=1}^J  r_j^D \log r_j^{-1}} \sum\limits_{n \in \mathbb{Z}} \int_0^{\log T} a_n e^{{\bf i} npx} dx= \frac{a_0}{\sum\limits_{j=1}^J  r_j^D \log r_j^{-1}},
\end{eqnarray*}
where the third equality we use the uniform convergence.
\section{Proof of Lemmas}\label{proofs}
\begin{proof}[Proof of Lemma \ref{yunus}]
Proof of a more general version of this lemma can be found in \cite{arsiv}. For convenience of the reader we repeat the main steps below, omitting the justification of technical details. We have
\[V_{\mathcal{T}}^{-} (\varepsilon)=\displaystyle \sum_{n=0}^{\infty}V_{T_n}^{-}(\varepsilon)  = \displaystyle \sum_{n=0}^{\infty} \sum_{w \in W_n} V_{r_w G}^{-}(\varepsilon),\]
where $r_w G$ is a copy of $G$ scaled by $r_w$. Recall that $V_G^-(\varepsilon)$ is given as in \ref{polinom}. A simple calculation shows
\[V_{rG}^{-}(\varepsilon)= \left \{
\begin{array}{cl}
  \displaystyle \sum_{m=0}^{d-1}\kappa_m(G) r^m \varepsilon^{d-m} &, \text{  for } \varepsilon < r g \\
  -r^d \kappa_d(G) &,\text{  for } \varepsilon \geq r g.
\end{array}
 \right.
\]

We calculate the Mellin transform of $\displaystyle V_{\mathcal{T}}^{-} (\varepsilon)/\varepsilon^d$: The Mellin transform $\mathfrak{M}[f;s]$ of $f:(0,\infty)\rightarrow \mathbb{R}$ is given by
$
\mathfrak{M}[f;s]=\int_0^{\infty} f(x) x^{s-1} dx.
$
A routine calculation shows that
\begin{equation}\label{mellin}
\mathfrak{M}\left[\frac{V_{rG}^{-}(\varepsilon)}{\varepsilon^d};s\right]=r^s  \sum_{m=0}^{d} \kappa_m(G) \frac{g^{s-m}}{s-m} ,\quad \text{ for } d-1<{\rm Re}(s)<d.
\end{equation}
Therefore, for $D<{\rm Re}(s)<d$,
\begin{eqnarray*}
\mathfrak{M}\left[\frac{V_{\mathcal{T}}^{-}(\varepsilon)}{\varepsilon^d};s \right]&=&\int_0^{\infty}  \sum_{n=0}^{\infty} \sum_{w \in W_n} V_{r_w G}^{-}(\varepsilon)\, \varepsilon^{s-d-1} ds=  \sum_{n=0}^{\infty} \sum_{w \in W_n} \int_0^{\infty} V_{r_w G}^{-}(\varepsilon)\, \varepsilon^{s-d-1} ds\\
&=&\sum_{n=0}^{\infty} \sum_{w \in W_n}  r_w^s  \left(\sum_{m=0}^{d} \kappa_m(G) \frac{g^{s-m}}{s-m}\right)=\zeta(s)\sum_{m=0}^{d} \kappa_m(G) \frac{g^{s-m}}{s-m}.
\end{eqnarray*}

Taking the inverse Mellin transform, we obtain
\[
\frac{V_{\mathcal{T}}^{-}(\varepsilon)}{\varepsilon^d}= \frac{1}{2 \pi {\bf i}}\int_{c-\bf{i} \infty}^{c+\bf{i} \infty} \left(\zeta(s)\sum_{m=0}^{d} \kappa_m(G) \frac{g^{s-m}}{s-m}\right)\, \varepsilon^{-s} ds,
\]
where $c$ satisfies $D<c<d$. Hence the claim is proved.
\end{proof}

In passing we note the following results which will be useful for us later. \begin{remark}\label{remark1}
Putting $r=1$ in \ref{mellin} gives
\begin{equation}\label{polynom}
\int_0^{\infty} V_G^- (\varepsilon) \, \varepsilon^{s-d-1} d \varepsilon= \sum_{m=0}^{d} \kappa_m \frac{g^{s-m}}{s-m} = g^{s-d} \sum_{m=0}^{d} \kappa_m \frac{g^{d-m}}{s-m}=g^{s-d}\frac{P(s)}{Q(s)},
\end{equation}
where $Q(s)=s(s-1)\cdots(s-d)$ is a polynomial of degree $d+1.$
A crucial observation is that the degree of the polynomial $P(s)$ is at most $d-1$. This is a consequence of the continuity of $V_G^-(\varepsilon)$ at $\varepsilon=g$: $\sum_{m=0}^{d-1} \kappa_m g^{d-m} = -\kappa_d$.

Therefore, for any $\sigma_1, \sigma_2$ with $d-1< \sigma_1 < \sigma_2  < d$, there exists $K>0$ such that \begin{equation}\label{toplamsinir}
\left \vert \sum_{m=0}^{d} \kappa_m \frac{g^{s-m}}{s-m} \right \vert \leq \frac{K}{|s|^2} \qquad \text{ for } \sigma_1 \leq {\rm Re}(s) \leq \sigma_2.
\end{equation}
\end{remark}
\begin{remark} \label{remark2}
Putting $s=D$ in \ref{polynom} gives \[\sum_{m=0}^{d} \kappa_m \frac{g^{D-m}}{D-m}= \int_0^{\infty} V_G^-(\varepsilon) \, \varepsilon^{D-d-1} d\varepsilon,\]
which shows that $\sum_{m=0}^{d} \kappa_m \frac{g^{D-m}}{D-m}$ can not be zero.
\end{remark}

\begin{proof}[Proof of Lemma \ref{strip}]
Recall that the contractivity ratios are assumed to be ordered as  $1>r_1\geq r_2\geq \cdots \geq r_J>0$. Let $D$ satisfy the Moran equation: $r_1^D+r_2^D+\cdots +r_J^D=1$. We define $\widetilde{D}$ to be the unique real solution of the equation $r_1^{\widetilde{D}}+r_2^{\widetilde{D}}+\cdots +r_{J-1}^{\widetilde{D}}=1$. It is clear that $\widetilde{D}<D$.

Let $f(s)=1-(r_1^s+r_2^s+\cdots +r_J^s)$, so that $\zeta(s)=1/f(s)$.
Let $s_0=\sigma_0+{\bf i}t_0$ be a zero of $f(s)$ in the strip
$\{s \,| \, \widetilde{D}< {\rm {Re}} (s) < D\}$. We will first show that $
{\rm {Re}}(r_j^{s_0}) \geq 0,$ for all $j=1,2,\cdots ,J :$

Since $s_0$ is a zero of $f$, we have $r_1^{s_0}+r_2^{s_0}+\ldots +r_J^{s_0}=1$. Taking real parts, we obtain
\begin{equation}\label{Real}
\textstyle \sum_{j=1}^{J} {\rm Re}(r_j^{s_0})=1.
\end{equation}
Suppose that for some $j_0$, we have ${\rm
{Re}}(r_{j_0}^{s_0})<0.$ Then,
\begin{equation*}
\sum_{j=1}^{J} {\rm Re}(r_j^{s_0}) < \sum_{j=1, j \neq j_0}^{J}
{\rm Re}(r_j^{s_0}) \leq \sum_{j=1,j \neq j_0}^{J} r_j^{\sigma_0}
\leq \sum_{j=1}^{J-1} r_j^{\sigma_0} < \sum_{j=1}^{J-1}
r_j^{\widetilde{D}}=1.
\end{equation*}
This contradicts to $(\ref{Real})$.

The nonnegativity of ${\rm Re}(r_j^{s_0})$ and \ref{Real} implies that
\begin{equation*}
{\rm {Re}}(f'(s_0))= \sum_{j=1}^{J} \log r_j^{-1}\ {\rm
{Re}}(r_j^{s_0}) \geq  \log r_1^{-1} \sum_{j=1}^{J} {\rm
{Re}}(r_j^{s_0}) = \log r_1^{-1} .
\end{equation*}
Thus, the zero of $f(s)$ (and therefore the pole of $\zeta(s)$) at $s=s_0$ is simple. Moreover,
\begin{equation*}
|{\rm{res}}(\zeta(s); s_0)|= \left | \frac{1}{f'(s_0)}\right | \leq
\left | \frac{1}{{\rm {Re}}(f'(s_0))} \right | \leq \frac{1}{\log
r_1^{-1}}.
\end{equation*}
\end{proof}

\begin{proof}[Proof of Lemma \ref{screen}]
With $\widetilde{D}$ as in Lemma \ref{strip},
choose $\sigma_L < D$ such that $\sigma_L > \max \{\widetilde{D},d-1 \} $.
Let $ r_1^{\sigma_L}+r_2^{\sigma_L}+\cdots +r_J^{\sigma_L}=:1+\lambda$. Let $0<\psi<\frac{\pi}{2}$ be chosen such that
\begin{equation}
3\,\psi\, \frac{|\log r_J |}{|\log r_1|}< \sqrt{\frac{\lambda}{1+\lambda}}
\label{conditon_beta}
\end{equation}
and let
\begin{equation}\label{mu}
\mu=\frac{\psi^2 r_J^D}{8}.
\end{equation}
Then there exists a unique real number $\sigma_R$ such that $r_1^{\sigma_R}+r_2^{\sigma_R}+\cdots +r_J^{\sigma_R}=1+\mu$. Note that
\[
\mu < \frac{\psi^2}{8}< 9 \psi^2 \frac{|\log r_J |^2}{|\log r_1|^2}<\frac{\lambda}{1+\lambda}<\lambda,
\]
hence $\sigma_L<\sigma_R$.

For $s=\sigma +{\bf i} t$, let $-\pi \leq \theta_j(t)< \pi$ be the angle of $r_j^s$: $\theta_j(t) \equiv t \log r_j $ (mod $2\pi$).

To construct the sequences $\{\alpha_k\}$ and $\{\beta_k\}$, we first determine the points $s=\sigma_R+{\bf i} t$ on the line $ {\rm Re} (s)=\sigma_R$ for which $|\theta_j(t)|<\psi$ for all $j=1,2,...,J$.

\begin{figure}[ht]
\centering
\includegraphics[scale=1]{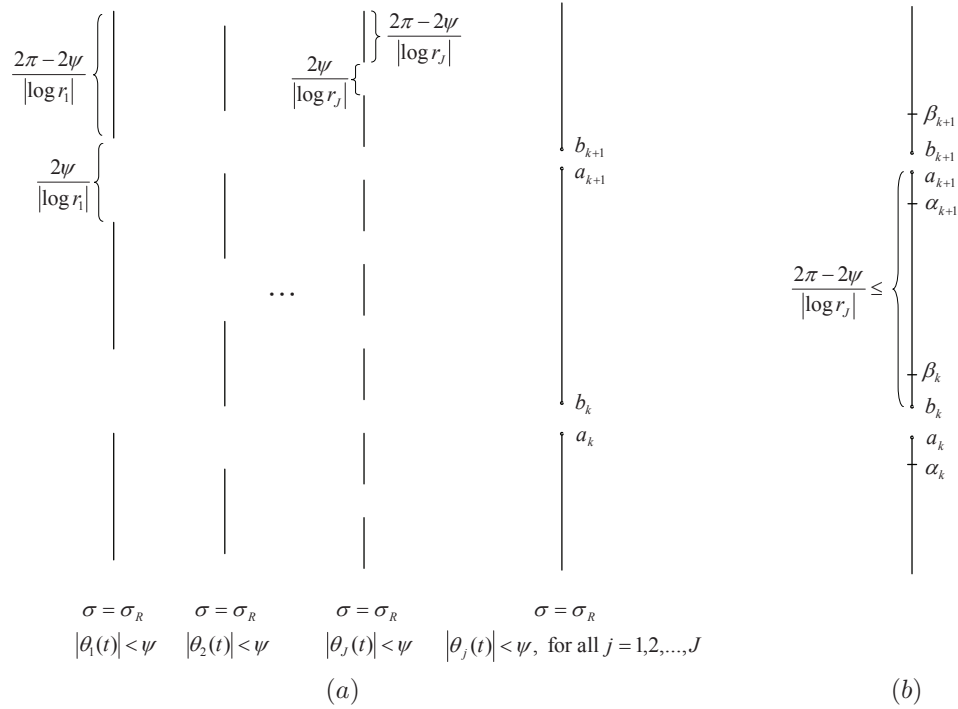}
\caption{Construction of the sequences $\{a_k\}$,  $\{b_k\}$, $\{\alpha_k\}$ and $\{\beta_k\}$.}\label{Fig:sigmas}
\end{figure}

The set $\{t \,| \ |\theta_j(t)|<\psi \, \text{ for all } j=1,2,\dots, J\}$ is a union of countably many disjoint open intervals. That is
\[
\{t \,| \, |\theta_j(t)|<\psi \, \text{ for all } j=1,2,\dots, J\}=\bigcup_{k \in \mathbb{Z}}I_k=: \bigcup_{k \in \mathbb{Z}}(a_k,b_k),
\]
with $|I_k|=b_k-a_k \leq 2 \psi/|\log r_J|$ and $a_{k+1}-b_k\geq (2 \pi - 2 \psi)/|\log r_J|$.

We define $\alpha_k:=a_k- 2 \psi/|\log r_1|$  and $\beta_k:=b_k+ 2 \psi/|\log r_1|$ (see Fig.\ref{Fig:sigmas}b). Clearly, $\alpha_k < \beta_k$. The inequality $\beta_k < \alpha_{k+1}$ follows from \[
2  \frac{2\psi}{|\log r_1|}<\frac{2 \pi - 2 \psi}{ |\log r_J |}.\] This inequality is a consequence of the following inequalities (the second one being \ref{conditon_beta}): \[
\psi\, \left(\frac{2 |\log r_J |}{|\log r_1|}+1\right)< 3\,\psi\, \frac{|\log r_J |}{|\log r_1|}<\sqrt{\frac{\lambda}{1+\lambda}}<1<\pi.
\]
Moreover $\alpha_{k+1}-\alpha_k = a_{k+1}-a_k >a_{k+1}-b_k\geq (2\pi-2\psi)/|\log r_J | > \pi/|\log r_J |$.

We will prove that $\zeta(s)$ is uniformly bounded on the (oriented) segments $\gamma_k^l, (l=1,2,3,4$ and $k\in \mathbb{Z})$ (see Fig. \ref{Fig:gamma}b). This will follow from the following estimates (recall that $f(s)=1-(r_1^s+r_2^s+\cdots +r_J^s)$ and $\zeta(s)=1/f(s)$):
\begin{enumerate}
\item[i)] ${\rm Re}(f(s))\geq \mu$ for $s \in \gamma_{k}^1$,
\item[ii)] ${\rm Im}(f(s))\leq -\sin \psi$ for $s \in \gamma_{k}^2$,
\item[iii)] ${\rm Re}(f(s))\leq-\frac{\lambda}{2}$ for $s \in \gamma_{k}^3$,
\item[iv)] ${\rm Im}(f(s))\geq \sin \psi$ for $s \in \gamma_{k}^4$.
\end{enumerate}

We begin with i): For $s=\sigma_R+{\bf i} t \in \gamma_k^1$, we have $|\theta_j(t)| \geq \psi$ for at least one $j=j_0$. Using \ref{mu} and the inequality $\cos \theta \leq 1-\theta^2/4$ for $-\pi/2 \leq \theta \leq \pi/2$, we get
\begin{eqnarray*}
{\rm Re}(r_{j_0}^s)&=& r_{j_0}^{\sigma_R} \cos \theta_{j_0}(t) <  r_{j_0}^{\sigma_R} \cos \psi \\
&\leq&   r_{j_0}^{\sigma_R} \left(1-\frac{\psi^2}{4}\right)= r_{j_0}^{\sigma_R} \left(1-\frac{2 \mu}{r_J^D}\right)
<  r_{j_0}^{\sigma_R} -2 \mu \frac{r_{j_0}^{D}}{r_J^D} \leq r_{j_0}^{\sigma_R}-2 \mu.
\end{eqnarray*}
Hence,
\begin{eqnarray*}
 {\rm Re} \left ( \sum_{j=1}^{J} r_j^s \right ) &=& {\rm Re} (r^s_{j_0})+ {\rm Re} \left ( \sum_{j=1,j \neq j_0}^{J} r_j^s \right )\\
 &\leq&  r_{j_0}^{\sigma_R}-2\mu+\sum_{j=1,j \neq j_0}^{J} r_j^{\sigma_R}=\sum_{j=1}^{J} r_j^{\sigma_R}-2\mu=1+\mu-2\mu=1-\mu.
\end{eqnarray*}
Therefore ${\rm Re}(f(s))\geq \mu$.

We now prove ii): We first show that for each $j=1,2,\dots,J$, $\theta_j(\alpha_k)$ satisfies $ \psi \leq \theta_j(\alpha_k)\leq \sqrt{\frac{\lambda}{1+\lambda}}$. Fix $j \in \{1,2,\dots,J\}$. By definition of $a_k,$ we have $-\psi \leq \theta_j(a_k)\leq \psi$. Therefore, there exists $m \in \mathbb{Z}$ such that $2\pi m - \psi \leq a_k \log r_j \leq 2 \pi m + \psi$. Then, since $\alpha_k=a_k- 2 \psi/| \log r_1|$,
\[
2 \pi m - \psi + \frac{2 \psi}{|\log r_1|} |\log r_j | \leq \alpha_k \log r_j \leq 2 \pi m + \psi + \frac{2 \psi}{| \log r_1|} |\log r_j|.
\]
Using \ref{conditon_beta} and noting that $| \log r_j| \geq | \log r_1 |$, we obtain
\[
2 \pi m + \psi \leq \alpha_k \log r_j \leq 2 \pi m + \sqrt{\frac{\lambda}{1+\lambda}}.
\]
Therefore $\psi \leq \theta_j(\alpha_k)\leq \sqrt{\frac{\lambda}{1+\lambda}}<1<\frac{\pi}{2} $.

Now, for $s=\sigma+ {\bf i} \alpha_k \in \gamma_k^2$, we have $\sigma_L \leq \sigma \leq \sigma_R<D$ and
\[
{\rm Im}(f(s))=-\sum_{j=1}^J r_j^{\sigma} \sin \theta_j(\alpha_k) \leq -\sin \psi \sum_{j=1}^J r_j^{\sigma}\leq -\sin \psi \sum_{j=1}^J r_j^{D}=-\sin \psi.
\]

We now prove iii): Reasoning as we did in the proof of part ii, it can be easily shown that  $-\sqrt{\frac{\lambda}{1+\lambda}} \leq \theta_j(\beta_{k}) \leq -\psi$ and for $\alpha_k \leq t \leq \beta_k$, we have $-\sqrt{\frac{\lambda}{1+\lambda}} \leq \theta_j(t) \leq \sqrt{\frac{\lambda}{1+\lambda}}$, for every $j \in \{1,2,\dots,J\}$. For $s\ \in \gamma_k^3$, we have $ s=\sigma_L + {\bf i} t ,\ \alpha_k \leq t \leq \beta_k $. Noting that $\cos \theta \geq 1-\theta^2/2$,
 \begin{eqnarray*}
 {\rm Re} (f(s)) =1-  \sum_{j=1}^{J} r_j^{\sigma_L} \cos \theta_j(t)&\leq& 1-\cos \left (  \sqrt{\frac{\lambda}{1+\lambda}} \right ) \sum_{j=1}^{J} r_j^{\sigma_L}\\
 &\leq &1-\left(1-\frac{1}{2}\frac{\lambda}{1+\lambda} \right)(1+\lambda)=-\frac{\lambda}{2}.
 \end{eqnarray*}

Finally, we prove the case iv). For $s \in \gamma_k^4$, we have  $s=\sigma + {\bf i} \beta_k,\ \sigma_L\leq \sigma \leq \sigma_R<D$ and $-\sqrt{\frac{\lambda}{1+\lambda}} \leq \theta_j(\beta_{k}) \leq -\psi$. Then \[ {\rm Im} (f(s)) =  -\sum_{j=1}^{J} r_j^{\sigma} \sin \theta_j (\beta_k) \geq -\sin (-\psi)  \sum_{j=1}^{J} r_j^{\sigma}\geq \sin \psi  \sum_{j=1}^{J} r_j^{D}=\sin \psi.
\]
\end{proof}

\begin{proof}[Proof of Lemma \ref{CauchyL}]

By \cite[Theorem 3.26]{LaFra}, there exists an increasing sequence $\{ \rho_n \}_{n=1}^{\infty}$
tending to infinity such that $\zeta(s)$ is uniformly bounded on the lines
${\rm Im}(s)= \pm \rho_n$. That is, there exists $K>0$ such that \[|\zeta(s)|\leq K,\  \ {\rm Im}(s)=\pm \rho_n, \ n=1,2,\dots.\]

By residue theorem,
\begin{eqnarray}\label{rezequation}
\frac{1}{2 \pi {\bf i}}\int_{c-{\bf i} \rho_n}^ {c+{\bf i} \rho_n}
\zeta_{\mathcal{T}}(s,\varepsilon) ds &=&\frac{1}{2 \pi {\bf i}} \int_{L_{n}}
\zeta_{\mathcal{T}}(s,\varepsilon) ds+ \frac{1}{2 \pi {\bf i}}\int_{\Gamma_{n}}
\zeta_{\mathcal{T}}(s,\varepsilon) ds \\ &+& \frac{1}{2 \pi {\bf i}} \int_{L^{\prime}_n}
\zeta_{\mathcal{T}}(s,\varepsilon) ds +\displaystyle \sum\limits_{\omega
\in \,\Omega_{\rho_n}\cap \mathfrak{D}} {\rm {res}}\left( \zeta_{\mathcal{T}}(s, \varepsilon);
\omega \right) \nonumber
\end{eqnarray}
where $L_n=\{ s \ | \ {\rm Im}(s)=\rho_n \} \cap \overline{\Omega}, \ L^{\prime}_n= \{ s\ |\ {\rm Im}(s)=-\rho_n \} \cap \overline{\Omega}, \ \Gamma_n=\{ s \ | \ -\rho_n \leq {\rm Im} (s) \leq \rho_n \} \cap \Gamma$ with appropriate orientations and $\Omega_{\rho_n}=\{ s \ | \ -\rho_n < {\rm Im} (s) < \rho_n \}\cap \Omega$.

Using \ref{toplamsinir} we obtain (for fixed $\varepsilon$)
\[
\left |   \int_{L_{n}} \zeta_{\mathcal{T}}(s,\varepsilon) ds \right | \leq   \int_{L_{n}} |\zeta_{\mathcal{T}}(s,\varepsilon)| |ds| \leq  K \int_{L_{n}}
\frac{K}{|s|^2 }|ds| \leq   \frac{K }{\rho_n^2 } (c-\sigma_L) \to 0,
\]
as $n \to \infty$, since the length of $L_{n}$ is at most $(c- \sigma_L)$ and $\rho_n \to \infty$ as $n \to \infty$. Similarly, $\lim_{n \to \infty}
\int_{L^{\prime}_{n}} \zeta_{\mathcal{T}}(s,\varepsilon) ds =0.$

It will be shown in the next lemma that the integral of
$\zeta_{\mathcal{T}}$ over $\Gamma$ absolutely converges, so letting $n \to \infty$ in \ref{rezequation} gives the desired result.
\end{proof}

\begin{proof}[Proof of Lemma \ref{CauchyGamma}]
Let $R_1 < 0 < R_2$ and let $\Gamma_{R_1}^{R_2}$ be the part of $\Gamma$ that lies in the strip $\{s\,|\, R_1 \leq {\rm Im} (s) \leq R_2 \}$. We will show that, for some $K>0$ (independent of $R_1$ and $R_2$), \[ \int_{\Gamma_{R_1}^{R_2}}\vert \zeta_{\mathcal{T}} (s,\varepsilon)\vert \vert ds \vert  \leq K\ \varepsilon^{d-\sigma_R},\quad \text{ for } \varepsilon <1. \]
Suppose $\Gamma_{R_1}^{R_2}$  have non-empty intersection with $\Gamma_k$ for $k_0 \leq k \leq k_1$. Then  \[
\int_{\Gamma_{R_1}^{R_2}}\vert \zeta_{\mathcal{T}} (s,\varepsilon)\vert \vert ds \vert \leq \sum\limits_{k=k_0}^{k_1} \int_{\Gamma_k}\vert \zeta_{\mathcal{T}} (s,\varepsilon)\vert \vert ds \vert=\sum\limits_{k=k_0}^{k_1}\sum\limits_{l=1}^{4}\int_{\gamma_k^l}\vert \zeta_{\mathcal{T}} (s,\varepsilon)\vert \vert ds \vert.
\]

Recall that, $|\zeta(s) \leq K$ for $s \in \Gamma$ (by Lemma \ref{screen}) and by \ref{toplamsinir}, \[\left \vert \sum_{m=0}^{d} \kappa_m \frac{g^{s-m}}{s-m} \right \vert \leq \frac{K}{|s|^2} \leq \frac{K}{\sigma_L^2+|{\rm Im}s|^2} \qquad \text{ for } \sigma_L \leq {\rm Re}(s) \leq \sigma_R.\]

Now, for $s \in \gamma_k^1$, we have $s=\sigma_{R}+{\bf i} t$, $\beta_{k-1} \leq t \leq \alpha_k$, and \[\int_{\gamma_k^1}\vert \zeta_{\mathcal{T}} (s,\varepsilon)\vert \vert ds \vert \leq K \, \varepsilon^{d-\sigma_R} \int_{\beta_{k-1}}^{\alpha_k} \frac{K}{\sigma_L^2+t^2} dt.
\]
Similarly,  $ \int_{\gamma_{k}^3} \vert \zeta_{\mathcal{T}} (s,\varepsilon)\vert \vert ds \vert \leq
 K \varepsilon^{d-\sigma_R} \int_{\alpha_{k}}^{\beta_{k}}
dt / (\sigma_L^2+t^2)$.
Therefore,
\begin{eqnarray}\label{limgamma13}
\sum_{k=k_0}^{k_1} \int_{\gamma_{k}^1+\gamma_k^3}
\vert \zeta_{\mathcal{T}} (s,\varepsilon) \vert \vert ds
\vert \leq K \varepsilon ^{d-\sigma_R} \int_{\beta_{k_0-1}}^{\beta_{k_1}} \frac{dt}{\sigma_L^2+t^2}\leq K
\varepsilon^{d-\sigma_R}.
\end{eqnarray}

By Lemma \ref{screen},  $\alpha_{k+1}-\alpha_k \geq \frac{\pi}{\log
r_J^{-1}}=:C$, $\alpha_0<0$ and $\alpha_1>0$.  This implies
$|\alpha_{k}| \geq (|k|-1)C$ and $|\beta_k| \geq (|k|-1)C$ for all $|k| \geq 2$.

Now for $s \in \gamma_{k}^2$, we have $s=\sigma+ {\bf i}
\alpha_{k}, \sigma_L\leq \sigma \leq \sigma_R$ and so for $|k| \geq 2$,
\begin{eqnarray*}
\int_{\gamma_{k}^2} \vert \zeta_{\mathcal{T}} (s,\varepsilon)
\vert \vert ds \vert \leq
K \varepsilon^{d-\sigma_R} \int_{\gamma_{k}^2}\frac{|ds|}{|s|^2} &\leq& K \varepsilon^{d-\sigma_R}  \frac{K}{|\alpha_{k}|^2} (\sigma_R-\sigma_L)
\leq \frac{K \varepsilon^{d-\sigma_R}}{(|k|-1)^2 C^2}.
\end{eqnarray*}
A similar inequality holds when $\gamma_k^2$ is replaced with $\gamma_k^4$. Hence
\begin{eqnarray}\label{limgamma24}
\sum_{k=k_0}^{k_1} \int_{\gamma_{k}^2+\gamma_{k}^4}
\vert \zeta_{\mathcal{T}} (s,\varepsilon) \vert \vert ds
\vert  \leq K\varepsilon^{d-\sigma_R}
\end{eqnarray}
Combining (\ref{limgamma13}) and (\ref{limgamma24}) we get the desired
result.
\end{proof}

\begin{proof}[Proof of Lemma \ref{lemma6}]
Let \[h(\varepsilon):= \varepsilon^{D-d} \sum \limits_{\omega \in \Omega \cap \mathfrak{D}} {\rm res}(\zeta_{\mathcal{T}} (s, \varepsilon);\omega)= \sum \limits_{\omega \in \Omega \cap \mathfrak{D}} \varepsilon^{D-\omega} {\rm res}(\zeta(s);\omega) \sum_{m=0}^d \kappa_m \frac{g^{\omega-m}}{\omega-m}. \]
By Lemma \ref{strip} and \ref{toplamsinir}
\begin{equation}\label{hepsilon} |h(\varepsilon)| \leq \sum\limits_{\omega \in \Omega \cap \mathfrak{D}} \frac{1}{|\log r_1}|\varepsilon^{D-{\rm Re}(\omega)} \frac{K}{|\omega|^2}= K \sum\limits_{\omega \in \Omega \cap \mathfrak{D}} \frac{\varepsilon^{D-{\rm Re}(\omega)}}{|\omega|^2}.
\end{equation}

Let $\Omega_n=\{ s \, | \, -n < {\rm Im}(s)<n \} \cap \Omega, \, n=1,2,\dots$. By \cite[Theorem 3.6]{LaFra}, there exists $C,M>0$ such that the number of poles of $\zeta(s)$ in the strip satisfies
\[
C n-M\leq \# \left( \mathfrak{D} \cap \{ -n <{\rm Im}(s)<n \} \right) \leq C n+M
\]
for all $n=1,2,\dots$. Let  $\Pi_1=\Omega_1$ and $\Pi_n=\Omega_n \setminus \Omega_{n-1}$ for $n\geq 2$. Then
\[
\#(\mathfrak{D} \cap \Pi_n) = \# (\mathfrak{D} \cap \Omega_n)\setminus (\mathfrak{D} \cap \Omega_{n-1})\leq C n + M - (C (n-1) -M)=C + 2 M
\]
for $n\geq 2$. For $n=1$, the above inequality clearly holds. Let $\eta >0$ be arbitrary. We will show that $|h(\varepsilon)|<\eta$ for sufficiently small $\varepsilon$: By \ref{hepsilon}, for any $n \geq 1$,
\begin{eqnarray}\label{Piler}
|h(\varepsilon)|\leq K \sum \limits_{\omega \in \Omega_n \cap \mathfrak{D}}\frac{\varepsilon^{D-{\rm Re}(\omega)}}{|\omega|^2} +
K \sum \limits_{\omega \in (\Omega\setminus\Omega_n) \cap \mathfrak{D}}\frac{\varepsilon^{D-{\rm Re}(\omega)}}{|\omega|^2}.
\end{eqnarray}
Now, for $\varepsilon<1$
\[ \sum \limits_{\omega \in (\Omega\setminus\Omega_n) \cap \mathfrak{D}}\frac{\varepsilon^{D-{\rm Re}(\omega)}}{|\omega|^2}
=\sum_{k=n+1}^{\infty} \ \sum \limits_{\omega \in \Pi_k \cap \mathfrak{D}}\frac{\varepsilon^{D-{\rm Re}(\omega)}}{|\omega|^2}
\leq \sum_{k=n+1}^{\infty} \frac{C+2M}{(k-1)^2},\]
since $\varepsilon^{D-{\rm Re}(\omega)}<1$ and for $\omega \in \Pi_k, (k \geq 2)$, we have $|\omega|^2\geq  |{\rm Im}(\omega)|^2\geq (k-1)^2$. Because of the convergence of the series $\sum_{k=2}^{\infty} \frac{1}{(k-1)^2}$, there exists $n_0$ such that the second term on the right-hand side of \ref{Piler} is less than $\eta/2$ for $n=n_0$. To deal with the first term, note that the set $\Omega_{n_0} \cap \mathfrak{D}$ has finitely many elements. Let $\delta:=\min \{ D-{\rm Re}(\omega)\, | \, \omega \in \Omega_{n_0} \cap \mathfrak{D}\}$. Recall that all the poles of $\zeta$, except the one at $s=D$, have real part less than $D$ (see \cite[Theorem 2.17]{LaFra}). Therefore $\delta>0$ and
\[
K \sum \limits_{\omega \in \Omega_{n_0} \cap \mathfrak{D}}\frac{\varepsilon^{D-{\rm Re}(\omega)}}{|\omega|^2}
\leq K \varepsilon^{\delta} \sum \limits_{\omega \in \Omega_{n_0} \cap \mathfrak{D}}\frac{1}{|\omega|^2}\leq K \varepsilon^{\delta}
\]
which is less than $\eta/2$ when $\varepsilon$ is sufficiently small.
\end{proof}

\bibliographystyle{amsplain}

\end{document}